\documentclass{amsart}
\usepackage{amssymb,latexsym,amscd}
\input xy 
\xyoption{all}
\numberwithin{equation}{section}

\newtheorem{theorem}{Theorem}[section]
\newtheorem{proposition}[theorem]{Proposition}
\newtheorem{conjecture}[theorem]{Conjecture}
\newtheorem{corollary}[theorem]{Corollary}
\newtheorem{remark}[theorem]{Remark}
\newtheorem{lemma}[theorem]{Lemma}
\newtheorem{example}[theorem]{Example}
\newtheorem{definition}[theorem]{Definition}

\newtheorem{algorithm}[theorem]{Algorithm}
\newtheorem{maintheorem}[theorem]{Main Theorem}
\newtheorem{claim}[theorem]{Claim}

\newtheorem{acknowledgements}[theorem]{Acknowledgements}
\newtheorem{condition}[theorem]{Condition}

\def\A{\mathcal{A}}
\def\AA{\mathbb{A}}

\def\ZZ{\mathbb{Z}}
\def\QQ{\mathbb{Q}}
\def\CC{\mathbb{C}}

\def\UU{\mathcal{U}}
\def\II{\mathcal{I}}
\def\JJ{\mathcal{J}}

\def\1{\mathbb{1}}

\def\gg{\mathfrak{g}}

\def\OO{\mathcal{O}}

\begin{document}

\title{Toric Poisson Ideals in Cluster Algebras}\author{Sebastian Zwicknagl}
              

\maketitle

\begin{abstract}
This paper investigates the Poisson geometry associated to a cluster algebra over
the complex numbers, and its relationship to compatible torus actions. We show, under some  assumptions,
that each Noetherian cluster algebra has only finitely many torus invariant Poisson prime ideals
and we show how to obtain using the exchange matrix of an initial seed. In
fact, these ideals are independent of the choice of compatible Poisson
structure.  In many interesting cases the ideals can be described more explicitly.
\end{abstract}

\tableofcontents  
\section{Introduction}

Cluster algebras were introduced by Fomin and Zelevinsky around the year 2000 (\cite{FZI}) in
order to understand the combinatorial properties of Lusztig's dual canonical
basis (see e.g. \cite{L1}
and \cite{L2}) in quantum groups. Being  commutative algebras,  cluster algebras relate to
these quantized function algebras via Poisson geometry--the {\it compatible
Poisson structures}  
introduced by Gekhtman, Shapiro and Vainshtein in \cite{GSV}, whose properties
are also the focus of their recent
book \cite{GSV B}.  Given a Noetherian Poisson algebra, it is a natural question to investigate the symplectic/Poisson geometry attached to it.     As there exists a natural algebraic torus $T$ which acts on the cluster algebra, we can follow the approach developed by Brown and Gordon  in \cite{Bro-Go} and Goodearl  in \cite{Goo1}.  The  first step, then,  would be to classify the torus orbits of Poisson ideals of the algebra,   for which we need to determine the torus invariant Poisson prime ideals, abbreviated as TPPs.

  In the present paper  we study TPPs in Noetherian (upper) cluster algebras using the combinatorial information obtained from the initial data, the seed of the cluster algebra. The main idea is that a cluster, and its nearby mutations should tell us  much about the geometry attached to the cluster algebra as a whole.

Cluster algebras are nowadays very well-established, hence we do not recall any of the definitions here, and refer the reader to the literature, resp.~our Section \ref{se:Cluster Algebras}. We will denote the initial seed by $({\bf x}, B)$ where ${\bf x}=(x_1,\ldots,x_n)$ and $B$ is an integer  $m\times n$-matrix with $m\le n$ such that its principal $m\times m$ submatrix is skew-symmetrizable.  The cluster variables $x_{m+1},\ldots, x_n$ are the frozen variables which we will call coefficients. The Poisson coefficient matrix $\Lambda$ is a skew-symmetric $n\times n$-matrix (see also Section \ref{se:poissonstructure}). We refer to a cluster algebra with compatible Poisson bracket as a {\it Poisson cluster algebra}, given by $({\bf x},B,\Lambda)$.

\begin{maintheorem}
\label{th:maintheorem}
Let $\AA$ be a Noetherian cluster algebra or upper cluster algebra over the complex numbers , given by $({\bf x},B,\Lambda)$,
and $T$ the torus of global toric actions.  Assume that it is sufficiently generic (for details see Section \ref{se:quotient cluster alg}). Then, there are only finitely many torus
invariant Poisson
prime ideals in $\AA$. 

\end{maintheorem}

Moreover, we consider cluster algebras for which the following  assumption  called COS (see Condition \ref{cond:COS} in Section \ref{se:COS}) holds:
 Let $\II, \JJ$ be TPPs, and let $codim(\II)=codim(\JJ)-k$. If $\II\subset \JJ$, then there exist TPPs $\II=\II_0,\II_1,\ldots,\II_{k-1},\II_k=\JJ$ such that  $\II=\II_0\subsetneq \II_1\subsetneq\ldots\subsetneq\II_{k-1}\subsetneq\II_k=\JJ$.
 Assuming COS, we can explicitly describe these ideals (Theorem \ref{th:ideals descr}), in terms of Laurent polynomials. It is known that COS  holds for many interesting classes of  cluster algebras, e.g. the algebra of functions on a complex semisimple algebraic group, introduced in \cite{BFZ}, or unipotent radicals, studied in \cite{GLS1} and \cite{GLS2}. Moreover, we conjecture that all cluster algebras with a compatible Poisson structure satisfy  COS (Conjecture \ref{con:COS}). As evidence, we show that acyclic cluster algebras of even rank, which could not satisfy COS if they contained non-trivial TPPs, do, indeed,  not have non-trivial TPPs (Theorem \ref{th:acyclic}). We can use this result to prove that the corresponding cluster variety is smooth. 
The reader should notice that if $\AA$ is the coordinate ring of an affine variety $X$, then $X$ does not necessarily equal the cluster manifold of \cite{GSV}, as $X$ may not be smooth (see Example \ref{ex:Singularities}) . The cluster manifold, however, is an open subset of $X$. 

Let us briefly explain why this result may be interesting. The canonical, resp.~dual canonical basis is not the only elusive feature about quantum groups. When ring theorists began to investigate the prime  spectra of quantum groups and to relate them to their classical counterparts-- symplectic leaves of the so called standard Poisson structure on semisimple complex algebraic groups-- it became clear that the following conjecture should hold (for notation and definitions see Goodearl's \cite{Goo1}):

\begin{conjecture}
\label{conj:homeo}
Let $G$ be a complex semisimple algebraic group and $\OO_q(G)$ the corresponding quantized function algebra. The topological space of primitive ideals, the primitive spectrum of $\OO_q(G)$, is homeomorphic to the space of symplectic leaves of the standard Poisson structure on $\CC[G]$, where the latter is endowed with the natural quotient topology. 
\end{conjecture}

The conjecture is an analogue of Kirillov's Orbit Method, resp.~Geometric Quantization (see \cite{Kir} and \cite{Kir2}). In this context, such a homeomorphism is referred to as a Dixmier-map. Hodges, Levasseur \cite{HL1}, \cite{HL2}, Toro \cite{HLT} and Joseph (\cite{Jo1},\cite{Jo}) constructed a stratification of the prime and primitive spectra of quantum groups into torus orbits. Very recently, Yakimov formulated a precise version of  Conjecture  \ref{conj:homeo} in \cite[Section 4]{Yak-spec}, where he employed results of Kogan and Zelevinsky \cite{KZ} in order to parametrize the symplectic leaves. Kogan and Zelevinsky's work foreshadows cluster algebras which appeared just a few  years later.
However, there is at present no  way to study the topology of these spaces. Using the methods developed in the proof of Theorem \ref{th:maintheorem} and Theorem \ref{th:ideals descr}, in particular by employing the concept of defining clusters, we hope to shed some light on the topology of the space of symplectic leaves. The recent progress regarding the theory of quantum cluster algebras (see e.g. Geiss, Leclerc and Schr\"oer's \cite{GLSq})  leads us to believe that  proving similar results for certain quantum cluster algebras will allow to establish continuity of  the Dixmier map for example in the case of  $SL_n$.
 Notice that we are not studying symplectic leaves in the present paper, but Poisson ideals which, in general, yield a coarser stratification--into symplectic cores rather than  symplectic leaves (see e.g. \cite{Bro-Go}).  However, we will explain in a future paper how we can apply our results to a  study of symplectic leaves.
   
 Let us briefly explain the organization of the paper, and some of the ideas of the proofs. First, we recall some definitions and well-known facts, about cluster algebras and their compatible Poisson structures. 
 The subsequent Section \ref{se:quotient cluster alg} is devoted to the proof of the main theorem. First, we consider the intersection $S$ of a toric Poisson prime ideal (TPP) $\II$ with a certain finite set of cluster variables $Y$, and show that $S$ must satisfy a number of conditions. The key observation is  Proposition \ref{pr:TPP super toric}, which states that the intersection of $\II$ with the polynomial ring $\CC[x_1,\ldots, x_n]$ generated by a cluster ${\bf x}=(x_1,\ldots,x_n)$ is  generated by a subset of $\{x_1,\ldots,x_n\}$.   We next introduce the notion of a {\it defining cluster} for a TPP $\II$ and construct such a cluster from a given cluster ${\bf x}$ through mutations. The defining clusters allow us to prove existence and finiteness results which complete the proof of Theorem  \ref{th:maintheorem}.  
As a  non-trivial application we show in Section \ref{se:acyclic cl} that the cluster variety defined by an acyclic cluster algebra of even rank with trivial coefficients is smooth (always under the assumption that  $B$ has full rank).  
   
In Section \ref{se:COS} we introduce  our strongest result. Suppose the cluster algebra, or upper cluster algebra $\AA$ satisfies the condition COS (see above). 
We can now explicitly describe the TPPs and their inclusion relations (Corollary \ref{co:incl}). Moreover, given an element $a\in\AA$ and a TPP $\II$ we can determine algorithmically whether $a\in \II$ (see Theorem \ref{th:ideals descr}). An appendix on torus invariant prime ideals completes the text.
Additionally, we explain our constructions on a running example, the Grassmannian $\CC[G(2,5)]$ of two-dimensional subspaces of $\CC^5$.
  
Clearly, this paper is only a starting point. It  suggests that if we manage to better understand the (Poisson) geometry associated with clusters and cluster algebras  , we should  be rewarded with important and   beautiful results. 

\section{Cluster Algebras}
\label{se:Cluster Algebras}
\subsection{Cluster algebras}
 
 In this section, we will review the definitions and some basic results on
cluster algebras. Denote by $\mathfrak{F}=\CC(x_1,\ldots,
x_n)$ the field of fractions  in $n$ indeterminates.  Let $B$ be a $m\times n$-integer matrix such that its principal $m\times m$-submatrix is skew-symmetrizable. Recall that a $m
\times
m$-integer matrix $B'$ is called skew-symmetrizable if
there exists a $m \times m$-diagonal matrix  $D$ with positive integer entries  
such that  $B' \cdot D$ is skew-symmetric.
We call the tuple $(x_1,\ldots,x_n, B)$ the {\it initial seed} of the cluster
algebra and   $ (x_1,\ldots x_m)$ a cluster, while ${\bf x}=(x_1,\ldots x_n)$ is
called an extended cluster. The cluster variables $x_{m+1},\ldots,x_n$ are called {\it coefficients}. We will now construct more  clusters, $(y_1,\ldots,
y_m)$ and extended clusters ${\bf
y}=(y_1,\ldots, y_n)$, which are transcendence bases of $\mathfrak{F}$, and the
corresponding 
seeds $({\bf y}, \tilde B)$ in the following way.

Define for each real number $r$ the numbers $r^+={\rm max}(r,0)$ and  $r^-={\rm
min}(r,0)$.
Given a skew-symmetrizable integer $m \times n$-matrix $B$, we define for each
$1\le i\le m$
the {\it exchange polynomial}
\begin{equation}
 P_i = \prod_{k=1}^n x_k^{b_{ik}^+}+ \prod_{k=1}^n  x_k^{-b_{ik}^-}\ .
\end{equation}

We can now define the new cluster variable $x_i'\in\mathfrak{F}$ via the equation
\begin{equation}
\label{eq:exchange}
 x_ix_i'=P_i\ . 
\end{equation}

This allows us to refer to the matrix $B$ as the {\it exchange matrix} of the
cluster $(x_1,\ldots,x_n)$, and to the relations  defined by Equation \ref{eq:exchange} for $i=1,\ldots,m$ as {\it exchange relations}. 

We obtain that $(x_1,x_2,\ldots, \hat x_i,x_i',x_{i+1},\ldots, x_n)$ is a
transcendence basis of $\mathfrak{F}$. We now define the new exchange matrix
$B_i=B'=(b_{ij}')$, associated to the new (extended) cluster $${\bf x}_i=(x_1,x_2,\ldots, \hat
x_i,x_i',x_{i+1},\ldots, x_n)$$
 by defining the coefficients $b_{ij}'$ as follows: 

$\bullet$ $b_{ij}' = -b_{ij}$ if $j \le n$ and $i = k$ or $j = k$,

$\bullet$ $b_{ij}' =  b_{ij} + \frac{|b_{ik} |b_{kj} + b_{ik} |b_{kj} |}{2}$ if
$j \le n$ and $i \ne k$ and $j \ne k$,

$\bullet$ $b_{ij}'=b_{ij}$ otherwise.

This algorithm is called  {\it matrix mutation}. Note that $B_i$ is again
skew-symmetrizable (see e.g.~\cite{FZI}). The process of obtaining a new seed is
called {\it cluster mutation}. The set of seeds obtained from a given seed $({\bf x},B)$ is  called the mutation equivalence class of  $({\bf x},B)$.

\begin{definition}
The cluster algebra $\mathfrak{A}\subset \mathfrak{F}$ corresponding to an
initial seed $(x_1,\ldots, x_n,B)$ is the subalgebra of $\mathfrak{F}$,
generated by the elements of all the clusters in the mutation equivalence class of $({\bf x},B)$ . We refer to the elements of the clusters as the {\it cluster variables}.
\end{definition}

\begin{remark}
Notice that  the coefficients, resp.~frozen variables $x_{m+1},\ldots, x_n$  
will never be mutated. Of course, that explains their name.
\end{remark}

We have the following fact, motivating the definition of cluster algebras in the
study of total positivity phenomena and canonical bases.

\begin{proposition} \cite[Section 3]{FZI}(Laurent phenomenon) Let $\mathfrak{A}$
be a cluster algebra with
initial extended cluster $(x_1,\ldots, x_n)$. Any cluster variable $x$ can be
expressed uniquely as a Laurent polynomial in the variables
$x_1,\ldots, x_n$ with integer coefficients.  
\end{proposition}

Moreover, it has been conjectured for all cluster algebras, and proven in many cases (see
e.g.~\cite{MSW} and \cite{FST},\cite{FT})  that the coefficients of these polynomials are
positive.

Finally, we recall the definition of the lower bound of a cluster algebra $\AA$
corresponding to a seed $({\bf x}, B)$. Denote by $y_i$ for $1\le i\le m$ the
cluster variables obtained from ${\bf x}$ through mutation at $i$; i.e., they
satisfy the relation $x_iy_i=P_i$.
 \begin{definition}\cite[Definition 1.10]{BFZ}
\label{def:lower bounds}
Let $\AA$ be a cluster algebra and let $({\bf x}, B)$  be a seed. 
The lower bound $ \mathfrak{L}_B \subset \AA$ associated with $({\bf x}, B)$ is the algebra
 generated by the set  $\{x_1,\ldots x_n,y_1\ldots, y_m\}$.
\end{definition}

\subsection{Upper cluster algebras}
\label{se:upper cluster algebras}
 Berenstein, Fomin and Zelevinsky
introduced the related concept of upper cluster algebras in \cite{BFZ}.  

\begin{definition}
  Let $\mathfrak{A} \subset \mathfrak{F}$ be a
  cluster algebra with initial cluster $(x_1, \ldots, x_n, B)$ and let, as above, $y_1,
  \ldots, y_m$ be the cluster variables obtained by mutation in the directions
  $1, \ldots, m$, respectively.
  
  \noindent(a) The upper bound $\UU_{{\bf x},B} ( \mathfrak{A})$ is defined as
  \begin{equation} 
\UU_{{\bf x},B} ( \mathfrak{A}) = \bigcap_{j = 1}^m \CC [x_1^{\pm 1}, \ldots
     x_{j - 1}^{\pm 1}, x_j, y_j, x_{j + 1}^{\pm 1}, \ldots, x_m^{\pm 1},
x_{m+1},\ldots,x_n] \ . \end{equation}

\noindent(b) The upper cluster algebra $\UU  ( \mathfrak{A})$ is defined as
$$\UU  ( \mathfrak{A})=\bigcap_{({\bf x'},B')}\UU_{\bf x'} ( \mathfrak{A})\ ,$$
where the intersection is over all seeds $({\bf x}',B')$ in the mutation equivalence class of $({\bf x},B)$.
\end{definition}

Observe that each cluster algebra is contained in its upper cluster algebra (see \cite{BFZ}). 
 
\subsection{The Standard Example}
\label{se:ex intro}
We will now introduce  our standard example--the coordinate ring $\CC[G(2,5)]$ of the
 Grassmannian $G(2,5)$, which is the variety of two-dimensional subspaces of $\CC^5$. We
define it as the subalgebra of the functions on $2\times
5$-matrices $\CC[Mat_{2,5}]$, generated by the ten $2\times2$-minors,
$$\Delta_{ij}=x_{1i}x_{2j}-x_{2i}x_{1j}$$
 with $1\le i<j\le 5$.
It is well-known that the minors are subject to the Pl\"ucker relations
\begin{equation}
\label{eq:Pluecker rel}
\Delta_{ik}\Delta_{j\ell}=\Delta_{ij}\Delta_{k\ell}+\Delta_{i\ell}\Delta_{jk}\ ,
\end{equation}
for $1\le i<j<k<\ell\le 5$.
The algebra $\CC[G(2,5)]$ has a natural cluster algebra structure (see
e.g.\cite{Scott}).  We can choose an initial seed with cluster variables $x_1=\Delta_{13}$ and $x_2=\Delta_{14}$ and coefficients $x_3=\Delta_{12}$, $x_4=\Delta_{23}$,
$x_5=\Delta_{34}$, $x_6=\Delta_{45}$, $x_7=\Delta_{15}$.
The corresponding  exchange matrix is:
\begin{equation}
\label{eq:ex-matrix ex}
  B= \left(\begin{array}{ccccccc}
{{\bf 0}}&{{\bf 1}}&-1&1&-1&0&0\\
{{\bf -1}}&{\bf 0}&0&0&1&-1&1\\
\end{array}\right)\  .\end{equation} 
The exchange relations are  therefore:
\begin{equation}
\label{eq:plucker ex}
 \Delta_{13} \ y_1= \Delta_{14}\Delta_{23}+\Delta_{12}\Delta_{34} \ ,
 \end{equation}
$$ \Delta_{14} \ y_2= \Delta_{34}\Delta_{15}+\Delta_{13}\Delta_{45} \ .$$
We observe from Equation \ref{eq:Pluecker rel} that $y_1=\Delta_{24}$ and
$y_2=\Delta_{35}$.  Indeed, the minors $\Delta_{ij}$, with $1\le i<j\le 5$, form the set of cluster
variables.  The cluster algebra is a cluster algebra of  finite type $A_2$ in
the classification of \cite{FZII}.

\subsection{Poisson structures}
\label{se:poissonstructure}
Cluster algebras are closely related to Poisson algebras. In this section we recall some of the related notions and results. 
 
\begin{definition}
Let $k$ be a field of charactieristic $0$. A Poisson algebra is a pair
$(A,\{\cdot,\cdot\})$ of a commutative $k$-algebra $A$ and a bilinear map
$\{\cdot,\cdot\}:A\otimes A\to A$,  satisfying for all $a,b,c\in A$:
\begin{enumerate}
\item skew-symmetry: $\{a,b\}=-\{b,a\}$ 
\item Jacobi identity: $\{a,\{b,c\}\}+\{c,\{a,b\}\}+\{b,\{c,a\}\}=0$,
\item Leibniz rule: $a\{b,c\}=\{a,b\}c+b\{a,c\}$.
\end{enumerate}
\end{definition} 
If there is no room for confusion we will refer to a Poisson algebra $(A,\{\cdot,\cdot\})$ simply as  $A$.

 Gekhtman, Shapiro and
Vainshtein showed in \cite{GSV} that one can associate Poisson structures to
cluster algebras in the following way.   Let $\mathfrak{A} \subset \CC[x_1^{\pm 1}, \ldots,
x_n^{\pm 1}] \subset \mathfrak{F}$ be a cluster algebra. A Poisson structure
$\{\cdot, \cdot\}$ on
$\CC [x_1, \ldots, x_n]$ is called log-canonical if   $\{
x_i,x_j\}=\lambda_{ij} x_ix_j$ with $\lambda_{ij}\in \CC$ for all $1\le i,j\le n$.
 
The Poisson structure can be naturally extended to $\mathfrak{F}$ by using the identity $0=\{f\cdot f^{-1},g\}$  for all $f,g\in\CC [x_1, \ldots, x_n]$.  We thus obtain that  $\{f^{-1},g\}=-f^{-2}\{f,g\}$ for all $f,g\in \mathfrak{F}$. 
We call $\Lambda=\left( \lambda_{ij}\right)_{i,j=1}^n$ the {\it coefficient matrix} of
the Poisson structure. We say that a Poisson structure on  $\mathfrak{F}$ is compatible with
$\mathfrak{A}$ if it is log-canonical with respect to each cluster $(y_1,\ldots,
y_n)$; i.e., it is log canonical on $\CC[y_1, \ldots, y_n]$.

\begin{remark}
 A  classification of Poisson structures compatible with cluster algebras was obtained by Gekhtman, Shapiro and Vainshtein in \cite[Theorem 1.4]{GSV}.  
\end{remark}

We will refer to  the cluster algebra $\AA$ defined by the initial seeed
$({\bf x},B)$ together with the compatible Poisson structure  defined by the coefficient 
matrix $\Lambda$ with respect to the cluster ${\bf x}$ as the {\it Poisson cluster algebra}
defined by the {\it Poisson seed}  $({\bf x},B,\Lambda)$.

It is not obvious under which conditions  a Poisson seed $({\bf x},B,\Lambda)$ would yield a Poisson bracket $\{\cdot,\cdot\}_\Lambda$   on   $\mathfrak{F}$ such that $\{\AA,\AA\}_\Lambda\subset \AA$. We have, however, the following fact. 

\begin{proposition} 
Let   $({\bf x},B,\Lambda)$ be a Poisson seed and $\AA$ the corresponding cluster algebra. Then $\Lambda$ defines a Poisson algebra structure on the upper bound $\UU_{{\bf x},B}(\AA)$ and the upper cluster algebra $\UU(\AA)$.
\end{proposition}

\begin{proof}
Denote as above by $\{\cdot,\cdot\}_\Lambda$ the Poisson bracket on $\mathfrak{F}$ by $\Lambda$.
  Observe  that  the algebras  $\CC[x_1^{\pm 1},\ldots x_{i-1}^{\pm 1}, x_i,y_i, x_{i+1}^{\pm 1}, \ldots, x_n^{\pm 1}]$   are Poisson subalgebras of the Poisson algebra $\CC[x_1^{\pm 1},\ldots x_n^{\pm 1}]$  for each $1\le i\le m$, as $\{x_i,y_i\}_\Lambda=\{x_i,x_i^{-1}P_i\}_\Lambda\in \CC[x_1,\ldots, x_n]$. If $A$ is a Poisson algebra and $\{B_i\subset A:i\in I\}$ is a family of Poisson subalgebras, then $\bigcap_{i\in I} B_i$ is a Poisson algebra, as well. The assertion follows.
\end{proof}

\subsection{Toric Actions}
\label{se:toric structure}
We recall the definitions and properties of local and global toric actions
from Gekhtman, Shapiro and Vainsthein \cite{GSV} (see also \cite{GSV B}) where they are introduced in the
context of cluster manifolds. As  discussed in \cite{GSV}, the cluster manifold  associated to a cluster algebra $\AA$ is not necessarily equal to the spectrum of maximal ideals of $\AA$, even when $\AA$ is Noetherian. For example the corresponding variety may have singularities  (see Example \ref{ex:Singularities}), and hence does not admit a manifold structure.
The main notions, however, carry over into our context.
 
 Let $X$ be an affine  variety such that $\mathfrak{A}=\CC[X]$  is a cluster
algebra or upper cluster algebra. Let ${\bf x}=(x_1,\ldots,x_n)$ be a cluster.  Following \cite[Section
2.3]{GSV} we define for each element
${\bf w}=(w_1,w_2,\ldots, w_n)\in \ZZ^n$ a {\it local toric action} of $\CC^{\ast}$ on
$\CC[x_1,\ldots, x_n]$
via maps $ \psi_{{\bf x},\alpha} : (x_1, \ldots, x_n) \mapsto (\alpha^{w_{ 1}}
x_1, \ldots, \alpha^{w_{n}} x_n)$ for all $\alpha \in \CC^{\ast}$. 
Assume now that we have chosen integer weights ${\bf w}_{\bf x}=(w_1,w_2,\ldots, w_n)$
for each cluster ${\bf x}$.  
The local
toric actions for two clusters are compatible if   the following diagram
commutes for any two clusters ${\bf x} = (x_1,  
  \ldots, x_n)$
and ${\bf y} = (y_1,\ldots, y_n)$, connected by a sequence of
mutations $T$:

$$\begin{xymatrix}{
\CC[{\bf x}]\ar[d]^{\psi_{{\bf x},\alpha}}\ar[r]^T&\CC[{\bf y}]\ar[d]^{\psi_{{\bf
y},\alpha}}\\
\CC[{\bf x}]\ar[r]^T& \CC[{\bf y}]
  }\end{xymatrix}\ .
   $$
 
Compatible local toric actions define a {\it global toric action} on the cluster
algebra and a {\it toric flow} on $X$. We have the following fact.

\begin{lemma}
  \cite[Lemma 2.3]{GSV} 
\label{le:compatible toric} Let $B$ denote the exchange matrix of the
  cluster algebra at the cluster ${\bf x}$. The local toric action  at ${\bf x}$
  defined by ${\bf w}\in \ZZ^n$  can be extended to a global toric action if and only
if $B
  \cdot { {\bf w}} = 0$. Moreover, if such an extension exists, it is unique.
\end{lemma}

We shall now discuss how to obtain all Poisson structures compatible with a
cluster algebra $\mathfrak{A}$, given a   Poisson  seed $({\bf x},B,\Lambda)$ where  $B$ is an  $m\times
 n$-matrix. Denote  $k=n-m$. Let $C$ be an integer $n
\times k$ matrix. We
define an action of the torus $(\CC^\ast)^k$ on $\CC[x_1,\ldots, x_n]$ where ${\bf d}=
(d_1 \ldots, d_k) \in (\CC^{\ast})^k$ acts on $x_i$, $1\le i\le n$, as
\begin{equation}
\label{eq:toric action in class}{\bf d}\cdot_C x_i = x_i \prod_{j = 1}^m d_j^{c_{ij}} . \end{equation}
The local toric action extends to a global toric action of $(\CC^{\ast})^k$ on
${\bf x}$ if and only if $B \cdot C = 0$ by Lemma \ref{le:compatible toric}.
Notice that every
skew-symmetric $k \times k$-matrix $V$ defines a Poisson bracket on.
$(\CC^{\ast})^k$ with $\{x_i, x_j \}_V = v_{ij} x_i x_j$. One
obtains the following result.

\begin{proposition}
  \label{prop:class POisson} \cite[Proposition 2.2]{GSSV}  Let $\UU(\AA)$ be the Poisson upper cluster algebra defined by $({\bf x},B,\Lambda)$, and denote by $\{\cdot,\cdot\}_\Lambda$ the Poisson bracket.
  Let $\{\cdot,\cdot\}'$ be another compatible Poisson structure and let $\{\cdot,\cdot\}'_\lambda$ be the bracket defined by $\{a,b\}'_\lambda=\lambda \{a,b\}'$. Then there exists  a $n \times k$-integer matrix $C$
defining a
  global toric action,  a   skew-symmetric $k \times k $ matrix $V$ and $\lambda\in \CC$ such that
  the  action of Equation \ref{eq:toric action in class} extends to a homomorphism of Poisson algebras
 $$((\CC^{\ast})^m, \{\cdot, \cdot\}_V) \times (\UU(\AA),
     \{\cdot, \cdot\}_\Lambda) \longrightarrow (\UU(\AA),\{\cdot, \cdot\}_\lambda' )\ .$$   
\end{proposition}

\subsection{Toric Actions on Subalgebras}
\label{se:torus on subs}
Let $B$ be an exchange matrix as above and let $T=ker(B)$. Let ${\bf i}=\{x_{i_1},\ldots, x_{i_k}\}$ be a $k$-element subset of ${\bf x}$ and let for $\ell=n-k$ be $\{x_{j_1},\ldots, x_{j_\ell}\}={\bf x}-{\bf i}$. Denote by $\ZZ^{\bf i}$ the sublattice of $\ZZ^n$ spanned by $e_{i_1},\ldots, e_{i_k}$ and by $T_{\bf i}$ the quotient $T/\ZZ^{\bf i}$. The global toric actions act on $\CC[x_{j_1},\ldots, x_{j_\ell}]$ as follows: Let $t\in T$ and $\alpha \in \CC^\ast$ then
$$t(\alpha)x_{j_h}=t_{\bf i}(\alpha)  x_{j_h}\ , $$
where  $t_{\bf i}$ denotes the image of $t$ under the natural projection of $T$ onto $T_{\bf i}$. Notice that if $B$ is generic, then $rank(T_{\bf i})=max(rank(T), n-|{\bf i}|)$.
\subsection{Compatible Pairs and Their Mutation}
\label{se:Compatible Pairs and Mut}
Section \ref{se:Compatible Pairs and Mut} is dedicated to compatible pairs and their mutation. As we shall see below, compatible pairs yield important examples of Poisson brackets which are compatible with a given cluster algebra structure. Note that our  
 definition is slightly different from the original on in \cite{bz-qclust}. Let, as above, $m\le n$. 
Consider a pair consisting of a skew-symmetrizable $m\times n$-integer matrix
$B$ with rows labeled by the interval $[1,m]=\{1,\ldots, m\}$ and columns  labeled by   $[1,n]$
 together with a skew-symmetrizable $n\times n$-integer matrix  $\Lambda$  with rows and
columns labeled by $[1,n]$.   

\begin{definition}
\label{def:compa pair}
Let $B$ and $\Lambda$ be as above. We say that the pair $(B,\Lambda)$ is
compatible if the coefficients $d_{ij}$ of the $m\times n$-matrix $D=B\cdot \Lambda$ satisfy
$d_{ij}=d_i\delta_{ij}$
for some positive integers $d_i$ ($i\in [1,m]$).  
\end{definition}
This means that $D=B\cdot \Lambda$ is a $m\times n$ matrix where the only
non-zero entries are positive integers on the diagonal of the principal $m\times
m$-submatrix.

The following fact is obvious.

\begin{lemma}
\label{le:full rank}
Let $(B,\Lambda)$ be a compatible pair. Then $B\cdot \Lambda$ has full rank.
\end{lemma}

  Let  $(B,\Lambda)$  be a compatible pair and let $k\in [1,m]$. We define for
$\varepsilon\in \{+1,-1\}$ a $n\times n$ matrix $E_{k,\varepsilon}$ via 
 
 \begin{itemize}
  \item $(E_{k,\varepsilon})_{ij}=\delta_{ij}$ if  $j\ne k$,
  
\item   $(E_{k,\varepsilon})_{ij}= -1$ if   $i=j= k$,
 
\item  $(E_{k,\varepsilon})_{ij}= max(0,-\varepsilon b_{ki})$ if  $i\ne j= k$.

\end{itemize}

 Similarly, we define a $m\times m$ matrix $F_{k,\varepsilon}$ via 
 
  \begin{itemize}
  \item $(F_{k,\varepsilon})_{ij}=\delta_{ij}$   if  $i\ne k$,
  
\item   $(F_{k,\varepsilon})_{ij}= -1$ if   $i=j= k$,
 
\item  $(F_{k,\varepsilon})_{ij}= max(0,\varepsilon b_{jk})$ if  $i= k\ne j$.

\end{itemize}

  We define a new pair $(B_k,\Lambda_k)$ as
  \begin{equation}
  \label{eq:mutation matrix and Poisson}
   B_k=   F_{k,\varepsilon}  B   E_{k,\varepsilon}\ , \quad
\Lambda_k=E_{k,\varepsilon}^T \Lambda E_{k,\varepsilon}\ ,
   \end{equation}
  where $X^T$ denotes the transpose of $X$. We have the following fact.
  
  \begin{proposition}\cite[Prop. 3.4]{bz-qclust}
  \label{pr:comp under mutation}
  The pair  $(B_k,\Lambda_k)$ is compatible. Moreover, $\Lambda_k$ is
independent of the choice of the sign $\varepsilon$. 
  \end{proposition}
 
 Proposition \ref{prop:class POisson} has the following obvious corollary.

\begin{corollary}
Let $\AA$ be a cluster algebra given by an initial seed $({\bf x}, B)$ where $B$ is a $m\times n$-matrix. If $(B,\Lambda)$ is a compatible pair, then $\Lambda$ defines a compatible Poisson bracket on $\mathfrak{F}$ and $\UU(\AA)$.
\end{corollary}

\begin{example}
 If $m=n$ (i.e.~there are no coefficients/frozen variables) and $B$ has full rank, then $(B, \mu B^{-1})$ is a compatible pair for all $\mu\in \ZZ_{> 0}$ such that $\mu B^{-1}$ is an integer matrix.
\end{example}
\begin{remark}
Another important example is the following.  Recall that double Bruhat cells in complex semisimple connected and simply connected algebraic groups have a natural structure of an upper cluster algebra (see \cite{BFZ}). Berenstein and Zelevinsky showed that the standard Poisson structure is given by compatible pairs relative to this upper cluster algebra structure (see \cite[Section 8]{bz-qclust}).  
\end{remark}

\subsection{Our running example}
\label{se:ex Poisson}
\subsubsection{The standard Poisson structure} In the case of $\CC[G(2,5)]$ we
have the  so called standard Poisson structure which is compatible with the
cluster algebra structure. It is most easily defined as the restriction of the standard Poisson bracket on  $\CC[Mat_{2,5}]$   to the subalgebra $\CC[G(2,5)]$: The standard Poisson bracket on $\CC[Mat_{2,n}]$ is given by
$$\{x_{ij},x_{k\ell}\}=\left(sgn(i-k)+sgn(j-\ell)\right)x_{i\ell}
x_{kj}\ ,$$
where $i,j\in\{1,2\}$ and $k,\ell\in[1,n]$ and $sgn$ denotes the sign function.

We observe that the Poisson bracket in the cluster 
$$\left(  \Delta_{13},\    \Delta_{14},\  \Delta_{12},\ \Delta_{23},\ \Delta_{34},\
\Delta_{45},\ \Delta_{15}\right)\ $$
is given by the matrix
\begin{equation}
\label{eq:Poisson ex matrix}
 \Lambda= \left(\begin{array}{ccccccc}
{\bf 0}&{\bf-1}&1&-1&-1&-2&-1\\
{\bf 1}&{\bf 0}&1&0&-1&-2&-1\\
-1&-1& 0&-1 &-2&-2 &-1 \\
1&0 &1 &0 &-1 &-2 &0 \\
1&1 &2 &1 &0 &-1 &0 \\
2&1 &2 &2 &1 &0 &1 \\
1&1 &1 &0 &0 &-1 &0 \\
\end{array}\right)\ . 
\end{equation}
It can be verified by direct computation that $(B,\Lambda)$ is a compatible
pair. 
\subsubsection{The toric actions}
\label{se:toric actions example}
The torus actions on the cluster algebra are given by the usual torus actions on
the Grassmannian, i.e.~the action of $(\CC^\ast)^2\times (\CC^\ast)^5$ on $M\in
Mat_{2,5}$ via left-, resp.~right-multiplication by diagonal matrices: 
$$ \left(\begin{array}{cc}
\ell_1&0\\
0&\ell_2\\
\end{array}\right)  \cdot  
M\cdot  \left(\begin{array}{ccccc}
r_1&0&0&0&0\\
0&r_2&0&0&0\\
0&0&r_3&0&0\\
0&0&0&r_4&0\\
0&0&0&0&r_5\\
\end{array}\right)  \ . $$

The weights of these actions on the initial cluster are:
$$wt(\ell_1)=wt(\ell_2)=(1,1,1,1,1,1,1) ,\  wt(r_1)=(1,1,1,0,0,0,1) ,\  
wt(r_2)=(0,0,1,1,0,0,0) \ , $$
$$wt(r_3)=(1,0,0,1,1,0,0),\  wt(r_4)=(0,1,0,0,1,1,0),\ wt(r_5)=(0,0,0,0,0,1,1)\ 
.$$
It is easy to verify that the weights span the kernel of $B$.

\section{Toric Poisson Prime Ideals}
\label{se:quotient cluster alg}

\subsection{The  Main Theorem}
 In this section we will prove the main result of the paper. The main theorem requires our seeds  to be generic in some sense. However, the parts of the proof each hold in some more generality, and we will need the more general versions later on. Therefore, we will introduce our restictions one by one throughout the section. First, we recall the definitions of Poisson, Poisson prime and toric Poisson prime ideals.
 
 \begin{definition} 
  Let $(A, \{\cdot, \cdot\})$ be a Poisson algebra over $\CC$ and $T$ an algebraic torus acting on $\AA$ by automorphisms.
  
  \noindent(a) A Poisson ideal
  $\II \subset A$ is an ideal which is closed under the Poisson bracket; i.e.,
  $\{h, a\} \in \II$ if $h \in \II$ and $a\in \AA$.

\noindent(b) A Poisson prime ideal is a Poisson ideal which is prime.
  
\noindent(c)  A toric Poisson prime ideal (TPP for short) is a Poisson ideal which is prime and torus invariant.
\end{definition}

    Now suppose that ${\bf x}=(x_1,\ldots,x_n)$ and that $B$ is a $m\times n$-integer matrix with skew-symmetrizable principal part (i.e.~the corresponding cluster algebra has rank $n$ and $n-m$ coefficients).  Let $\Lambda$ be a skew-symmetric integer matrix such that $({\bf x}, B,\Lambda)$ defines a Poisson cluster algebra. We have the following main result (the precise statement follows in Theorem \ref{th:Stratification total 2}).
\begin{maintheorem}
\label{th:Stratification total} 
Let $\mathfrak{A}$ be a Noetherian Poisson cluster algebra  over the complex numbers, corresponding to the triple
$({\bf x}, B,\Lambda)$ introduced above. Let $X$ be an affine variety such that $\mathfrak{A}=\CC[X]$ is its coordinate ring. If  the seed is generic (for precise definitions see below),  then  $\AA$  contains only finitely many Poisson prime ideals which are invariant under the global toric actions.  
\end{maintheorem}

\begin{remark}
 The proof also works if $\AA$ is a Noetherian upper cluster algebra or an upper bound. 
\end{remark}

\begin{remark}
Recall that if  $(B,\Lambda)$ is a compatible pair, then $B\cdot \Lambda$ has full rank (see Lemma \ref{le:full rank}).
\end{remark}


We observe the following fact.

\begin{proposition}
Let $\AA$ be a Noetherian Poisson cluster algebra given by $({\bf x},B,\Lambda)$, and let  $({\bf x},B,\Lambda')$ define another compatible Poisson structure. If $\II$ is a TPP for $({\bf x},B,\Lambda)$, then it is also a TPP for  $({\bf x},B,\Lambda')$.
\end{proposition}

\begin{proof}
By Proposition \ref{prop:class POisson} it suffices to show that if $t\in T$ is an element of the torus of global toric actions  and $f\in \II$, then $t.f\in \II$, as well. But this is the  definition of torus invariant.
\end{proof}

Additionally, notice the following fact.
\begin{lemma}
Let $\mathfrak{A}$ be a Noetherian cluster algebra over the complex numbers, and
let $\{\cdot,\cdot\}$ be a compatible Poisson structure. Then, the global toric
actions are compatible with the Poisson structure, i.e.
$$\{t.x,t.y\}=t.\{x,y\}\ ,\quad \rm{for\  all }\ x,y\in \AA\ .$$ 
\end{lemma}

\begin{proof}
Express $x$ and $y$ as Laurent polynomials in a cluster. 
The assertion is  now proved by straightforward computation, using the fact that
the Poisson bracket is log-canonical in the cluster variables.
\end{proof}

   Let us return to   Theorem \ref{th:Stratification total} which  we will prove  in four steps which contain several independent and important results.

 \subsection{Step 1: The set $TP_Y$}  

Let ${\bf x}$ be a cluster. We will, for convenience, also use the notation ${\bf x}$ to denote the set ${\bf x}=\{x_1,\ldots,x_n\}$.  Denote, as before, by $y_1, \ldots,
y_m$ the
cluster variables obtained by mutation in the directions $1, \ldots,
m$, respectively. Let 
$$Y=\{x_1, \ldots, x_n,y_1,\ldots, y_m,\}\cup\{1\}\ .$$


 Denote by $I_S$ the Poisson ideal generated by a subset $S\subset Y$. Notice that if $1\in S$ then, of course, $I_S=\AA$. Denote by $J_S\subset \CC[x_1,\ldots, x_n]$ the ideal generated by $S\subset {\bf x}$ in $\CC[x_1,\ldots, x_n]$. The ideal $J_S$ is torus invariant Poisson and prime in the polynomial ring $\CC[x_1,\ldots, x_n]$. Denote by $M_S$ the monoid consisting of all  monomials in $\{x_{j_1},\ldots, x_{j_\ell}\}={\bf x}-S$; i.e., a typical element of $M_S$ is a monomial $x^{\beta}$ where $0\ne \beta\in \ZZ_{\ge 0}^n$ and $\beta_i=0$ if $x_i\in S$.

We define a subset $TP_Y$ of the power set of $Y$ as follows. 

\begin{definition}
\label{def:TPP}
The set $TP_Y$ is  the set whose elements are the subsets of $S\subset Y$ such that  

\begin{enumerate}
\item $I_S\cap Y=S$,
\item $I_S\cap\CC[x_1,\ldots,x_n]=J_S$.

\end{enumerate} 
\end{definition}

 Definition \ref{def:TPP} implies the following lemma.

\begin{lemma}
\label{le:TP implies}
Let $S\in TP_Y$. Then,
\begin{enumerate}
\item $I_S\cap M_S=\emptyset$, and 
\item  $x_i\in S$ or $y_i\in S$ is equivalent to  $P_i\in J_S$.
\end{enumerate}
\end{lemma}
 
\begin{remark}
When determining the elements of $TP_Y$, it is clearly most difficult to verify that $I_S\cap Y=S$. However, we conjecture (see Conjecture \ref{conj:poisson and prime}) that we can solve this problem combinatorially, using the concept of defining clusters  introduced in Section \ref{se:existence}.  
\end{remark}

\subsection{Step 2:$TP_Y$ characterizes TPPs}

 We begin with the following key result.
 Recall the notation $ker(B)=T$.
\begin{proposition} 
\label{pr:TPP super toric}
Let ${\bf x}$ be a cluster,  $rank(T+Im(\Lambda))=n$  and $\II$ be a non-zero TPP. Then the ideal $\II$ contains a cluster variable $x_i\in {\bf x}$. 
\end{proposition}

\begin{remark}
The condition $rank(T+Im(\Lambda))=n$ is satisfied if $(B,\Lambda)$ is a compatible pair. Indeed, by definition $rank(B\cdot \Lambda)=m$, resp.~$rank(B(Im(\Lambda))=m$, and, hence, $rank(T+Im(\Lambda))=n$. 
\end{remark}
\begin{proof}
Notice first that $\II_{\bf x}\ne 0$. Indeed, let $0\ne f\in \II$. We can express $f$ as a Laurent polynomial in the variables $x_1,\ldots, x_n$; i.e., $f=x_{1}^{-c_1}\ldots x_n^{-c_n} g$ where $c_1,\ldots, c_n\in \ZZ_{\ge 0}$ and $0\ne g\in\CC[x_1,\ldots, x_n]$. Clearly, $g= x_{1}^{c_1}\ldots x_n^{c_n}f \in \II_{\bf x}$. Observe, additionally, that $\II_{\bf x}$ is prime and torus invariant.

We complete the proof  by contradiction. Let $f=\sum_{{\bf w}\in \ZZ^n} c_{\bf w} x^{\bf w}\in\II_{\bf x}$ and suppose that $f$ cannot be factored into $f=gh$ with $g\in\II_{\bf x}$ or $h\in\II_{\bf x}$. We have to show that $f=x_i$ for some $i$. Since the ideal is prime, it suffices to show that $f$ is a monomial. We assume that $f$ has the smallest number of nonzero summands such that no monomial term $c_{\bf w} x^{\bf w}$ with $c_{\bf w}\neq 0$ is contained in $\II$. It must therefore have at least two monomial terms.

We need the following fact.

\begin{lemma}
Using the notation introduced above,   a monomial $x^{\bf w}$ with ${\bf w}\in \ZZ^n$ is torus invariant if and only if ${\bf w}\in T^\top$, where $^\top$ denotes the orthogonal complement with respect to the standard bilinear form on $\ZZ^n$.
\end{lemma}
\begin{proof}
Recall from Section \ref{se:toric structure} that if ${\bf b}\in T$ defines a global toric action $ \psi_{{\bf x},\alpha}$, then $x^{\bf w}$ is  invariant under $ \psi_{{\bf x},\alpha}$ if and only if $\sum_{i=1}^n w_ib_i=0$. The assertion follows.
\end{proof}
  
 The function $f$, considered above, must be torus invariant, 
 hence for each pair ${\bf w},{\bf w}'\in \ZZ^n$ with $c_{\bf w},c_{{\bf w}'}\ne 0$ we obtain that ${\bf w}-{\bf w}'={\bf v}\in T^\top$.  Denote by $rad(\Lambda)$ the radical of the skew-symmetric bilinear form, i.e.~the set of ${\bf u}\in \ZZ^n$ such that ${\bf u}^T\cdot \Lambda\cdot{\bf u}'=0$ for all ${\bf u}'\in \ZZ^n$. We have the following fact.

\begin{lemma}
The intersection $rad(\Lambda)\cap T^\top=\{0\}$. 
\end{lemma}
\begin{proof}
  
We know that $Im(\Lambda)\otimes\QQ+ker(B)\otimes \QQ= \QQ^n$. Hence $Im(\Lambda)^\top\otimes\QQ\cap T\otimes \QQ^\top=\{0\}$. Notice, that by definition $Im(\Lambda)^\top\otimes\QQ=rad(\Lambda)\otimes \QQ$. The assertion is proved.
\end{proof}

   Assume as above that $c_{\bf w},c_{{\bf w}'}\ne 0$ and ${\bf w}-{\bf w}'={\bf v}\in T^\top$.  The previous lemma yields that ${\bf v}\notin rad(\Lambda)$. This implies that there exists $i\in [1,n]$ such that $\{x_i,x^{\bf v}\}\ne 0$. Therefore,  $\{x_i, x^{\bf w}\}=cx_ix^{\bf w}\ne dx_ix^{{\bf w}'}\{x_i, x^{{\bf w}'}\}$ for some $c,d\in \CC$. 

Clearly, $cx_if-\{x_i,f\}\in \II$ and 
$$ cx_if-\{x_i,f\}=(c-c)c_{\bf w} x^{\bf w}+(c-d)c_{{\bf w}'} x^{{\bf w}'}+\ldots\ . $$

Hence, $cx_if-\{x_i,f\}\ne 0$ and it has fewer monomial summands than $f$ which contradicts our assumption. Therefore,   $\II$ contains a  monomial, and because it is prime it must contain some $x_i\in {\bf x}$. 
  The proposition is proved.
\end{proof}

Let ${\bf x}$ be a cluster in $\AA$ and ${\bf i}=\{x_{i_1},\ldots, x_{i_k}\}$ be a $k$-element subset of ${\bf x}$. Moreover, denote by ${\bf j}=\{x_{j_1},\ldots, x_{j_\ell}\}={\bf x}-{\bf i}$ with $\ell=n-k$   and by $\Lambda_{\bf i}$ the submatrix of $\Lambda$  obtained by removing the rows and columns labeled by $\{i_1,\ldots, i_k\}$. Recall the notation $T_{\bf i}$ as introduced in Section \ref{se:torus on subs}. Observe that $Im(\Lambda_{\bf i})$ and $T_{\bf i}$ can be naturally viewed as sublattices of the lattice $\ZZ^{\bf j}\subset \ZZ^n$ generated by the standard basis vectors $e_{j_1},\ldots, e_{j_\ell}$. We need our first major  condition.
\begin{condition}
Let $B$ and $\Lambda$ be as above. The cluster ${\bf x}$ is super-toric, if $rank(T_{\bf i}+ Im(\Lambda_{\bf i}))=n-k=\ell$ for all subsets ${\bf i}\subset [1,n]$.
\end{condition}
  
\begin{remark}
Notice that if $B$ and $\Lambda$ are generic, and $m<n$, then the corresponding cluster will be super-toric. Indeed, for generic matrices we have $$rank(Im(\Lambda_{\bf i}))=min(rank(\Lambda_{\bf i}), 2 \lfloor\frac{n-k}{2}\rfloor)$$ while $rank(T_{\bf i})=min(rank(T), n-k)$.  
\end{remark}  
  
\begin{theorem}
\label{th:intersection}
Let $\AA$ be a Noetherian Poisson cluster algebra and ${\bf x}$ a super-toric cluster, and let $\II$ be a TPP. Then the set $\II\cap Y=S$ is an element of   $TP_Y$.
\end{theorem}

 \begin{proof}

The first condition of Definition \ref{def:TPP} is obvious. The second, however, is a bit more interesting. It follows from a  stronger version of Proposition \ref{pr:TPP super toric}.
\begin{proposition} 
\label{pr:TPP super toric +}
Let ${\bf x}$ be a super-toric cluster, and $\II$ be a non-zero TPP. Then the ideal $\II_{\bf x} =\II\cap\CC[x_1,\ldots, x_n]$ is generated by a non-empty subset of ${\bf x}$. 
\end{proposition}

\begin{proof}
 Suppose that  $\II\cap{\bf x}= {\bf i}=\{x_{i_1},\ldots, x_{i_k}\}$. Suppose that $f\in \CC[x_{j_1},\ldots, x_{j_\ell}]\in \II$. The cluster, however,  is super-toric, hence we can adapt the argument from the proof of Proposition \ref{pr:TPP super toric} to $\CC[x_{j_1},\ldots, x_{j_\ell}]$, $T_{\bf i}$ and $\Lambda_{\bf i}$ and show that there exists some $x_j\notin {\bf i}$ such that $x_j \in \II$. We obtain the desired contradiction and the proposition is proved. 
\end{proof}

Theorem \ref{th:intersection} is proved. 
\end{proof}

  Proposition  \ref{pr:TPP super toric +} has the following immediate consequence. Denote by $X_\II$   the zero locus of an ideal $\II$ in $X$.
 
 \begin{proposition}
 \label{pr:dimension form}
 Let $\II\subset \AA$ be a TPP, ${\bf x}$   a super-toric cluster and let $|\II\cap{\bf x}\}|=k$. 
  Then $dim(X_\II)\ge n-k=\ell$. 
 \end{proposition} 
 \begin{proof}
 Recall that we denote by $\{x_{j_1},\ldots, x_{j_\ell}\}$ the set ${\bf x}-S$.
 Observe  that $\{x_{j_1},\ldots, x_{j_\ell}\}$ is algebraically independent over $\AA/\II$ by Proposition \ref{pr:TPP super toric}.  Moreover, no Laurent polynomial  $f=x^{-{\bf w}}g$ with ${\bf w}\in \ZZ^{n}_{\ge0}$ and $g\in \CC[x_{j_1},\ldots, x_{j_\ell}]$ is contained in the ideal. Hence, the field of fractions $\CC(x_{j_1},\ldots, x_{j_\ell})$ embeds into the field of fractions of  $\AA/\II$ and the assertion follows. 
\end{proof}

\subsection{Step 3: Existence of TPPs for $S\in TP_Y$}
 \label{se:existence}
In order to study a  TPP $\II$ we need to capture as much information as possible in the set ${\bf x}\cap \II$. For this reason we prefer to work with clusters such that the cardinality  of ${\bf x}\cap \II$ equals the dimension of $X_\II$.
 We therefore introduce the notion of   {\it defining clusters}. Their properties will be further investigated in the following Section \ref{se:finiteness and defining clusters}.
 \begin{definition}
 \label{def:defining clusters}
(a) Let $\AA$ be a Poisson cluster algebra and  $\II\subset \AA$ a TPP. A cluster ${\bf x}$ is called a defining cluster for   $\II$  if $x_i\in \II$ implies that $y_i\in \II$, as well.

\noindent(b) A set $S\in TP_Y$ is called defining, if $x_i\in S$ implies that $y_i\in S$, as well.
\end{definition}
 The main result of this section is the following theorem.
\begin{theorem}
\label{th:TPP existence}
Let $\AA$ be as above,  and let $S\in TP_Y$ be defining. Then there exists a TPP $\II\subset \AA$ such that $\II\cap Y=S$.
\end{theorem}

\begin{proof}
 
We first introduce the notion of a Poisson multiplicative set. Let $\AA$ be a Poisson algebra and $S\subset \AA$ a multiplicative set (i.e., $s\cdot t\in S$ for all $s,t\in S$). We call $S$ a {\it Poisson multiplicative set} if $\{s,t\}\in  S\cup\{0\}$ for all $s,t\in S$. We have the following fact.

\begin{lemma}
Let $\AA$ be a Poisson algebra and let $S\subset \AA$ be a Poisson multiplicative set.
\noindent(a) The algebra $\AA[S^{-1}]$ is a Poisson algebra with Poisson bracket 
$$\{s^{-1} f,t^{-1} g\}= s^{-2}t^{-2}\{s,t\} fg - s^{-2}t^{-1}\{s,g\}f-  s^{-1}t^{-2} \{f,t\}g+s^{-1}t^{-1}\{f,g\}\ ,$$
for all $s,t\in S$ and $f,g\in \AA$.

\noindent(b) The natural embedding of algebras $\AA\hookrightarrow \AA[S^{-1}]$ is a homomorphism of Poisson algebras.

\noindent(c) Extension and Contraction define maps from the set of Poisson ideals $\II$ in $\AA$ such that $\II\cap S=\emptyset$  to the set of Poisson ideals in  $\AA[S^{-1}]$. Moreover, if restricted to prime ideals, this map becomes a bijection.
\end{lemma}

\begin{proof}
Part (a) can be proved by direct  computation using the identity
$\{s^{-1},f\}=-s^{-2}\{s,f\}$ for all $s\in S$ and $f\in \AA$.  It follows from
$$0=\{1,f\}=\{s^{-1}s,f\}=s^{-1}\{s,f\}+s\{s^{-1},f\}\ .$$

For part (b) it suffices to observe that $\AA\subset \AA[S^{-1}]$ is a Poisson subalgebra. Part(c) is also immediate from standard localization theory and the fact that if $B\subset A$ is a Poisson subalgebra of a Poisson algebra $A$ and $J\subset A$ a Poisson ideal, then the intersection $J\cap B$ is a Poisson ideal in $B$.  The lemma is proved.
\end{proof}

Recall that we denote by $\{x_{j_1},\ldots, x_{j_\ell}\}$ the set ${\bf x}-S$.
Observe that the set  $\{ \gamma x_{j_1}^{\gamma_1}\ldots, x_{j_\ell}^{\gamma_\ell}:\gamma\in \CC;\alpha_1,\ldots,\alpha_\ell\in \ZZ_{\ge 0}\}$ is a Poisson multiplicative set. Consider the ideal $\hat I_S$  generated by $S $ in $ \AA[x_{j_1}^{-1},\ldots, x_{j_\ell}^{-1}]$. It is proper, hence contained in a minimal prime ideal (recall that an ideal $P$ in a ring $R$ is called a minimal prime  over an ideal $I$ if $P/I$ is a minimal prime in $R/I$). It is a Poisson prime ideal  by the following fact.

\begin{lemma}\cite[Lemma 6.2]{Goo1}\label{le:Poisson core primes}
Let $\AA$ be a Poisson algebra over $\CC$ and let $\II\subset \AA$ be a Poisson ideal. Then all minimal prime ideals over $\II$ are Poisson ideals. 
\end{lemma}
 The ideal $\hat I_S$  is torus invariant by Lemma \ref{le:primes are H primes}.  Its intersection with $\AA$ yields a TPP $\II$. By construction $\II\cap {\bf x}=S\cap {\bf x}$.  We now obtain from  Proposition \ref{pr:TPP super toric} that if an exchange polynomial $P_i\in \II$ (and, therefore, $P_i\in J_S$) then  $y_i\in S$, because $S$ is defining.  We conclude that $\II\cap Y=S$. Theorem \ref{th:TPP existence}  is proved.   
\end{proof} 
\subsection{Step 4: Finiteness of the Stratification}

In order to prove the finite-ness of the stratification we need to make another assumption, in some sense saying that the cluster variables are generic in a geometric way. This assumption appears to be the hardest to verify for a given cluster algebra.

\begin{condition}  Let $\AA$ and ${\bf x}$ be as above.

\noindent(a)   We say that ${\bf x}$ is geometrically generic if for all $S\in TP_Y$ for which ${\bf x}$ is defining and all minimal Poisson primes $P$ over $S$ we have $dim(\AA/P)=n-|{\bf i}|$ where $ {\bf i}=S\cap{\bf x}$.  

\noindent(b) Let $r\in \ZZ_{\ge 0}$. We say that ${\bf x}$ is geometrically $r$-generic if each cluster that can be reached from $r$ with at most $r$ mutations is geometrically generic. 
\end{condition}

\begin{remark}
A cluster ${\bf x}$  is geometrically generic if for each  minimal Poisson prime $P$ over some set $S\in TP_Y$ for which ${\bf x}$ is defining, there exist prime ideals (not necessarily Poisson) $0=\II_0\subsetneq \II_1\subsetneq\ldots\subsetneq \II_k=P$ such that $\II_j\cap{\bf x}=\{x_{i_1},\ldots,x_{i_j}\}$ for all $1\le j\le k$.  
\end{remark}
\label{se:finiteness and defining clusters}
Recall that $m\le n$ is the number of cluster variables in each cluster. 
\begin{proposition}
\label{th:finiteness}
Let $\AA$ and ${\bf x}$ be as above, and let $S\in TP_Y$. If ${\bf x}$ is geoemtrically $m$-generic, then there exist only finitely many many TPPs $\II_1,\ldots, \II_\ell$ such that $\II_j\cap Y=S$.
\end{proposition}
\begin{proof}

 We first need some properties of geometrically generic clusters.

\begin{lemma}
\label{pr:defining clusters dim}
Suppose that ${\bf x}$ is geometrically generic  and a defining cluster for a TPP $\II\subset \AA$ with $\II\cap Y=S$.
Then $|{\bf x}-S|= dim(X_\II)$. 
\end{lemma}
\begin{proof}
 Proposition \ref{pr:dimension form} yields that $|{\bf x}-S|\le dim(\AA/\II)$.  Equality follows from the assumption that the cluster is geometrically generic. The lemma is proved.
 \end{proof}

Starting from any cluster ${\bf x}$ we can construct a defining cluster for a TPP $\II\subset \AA$ using the following algorithm. 

\begin{algorithm}
\label{alg:q-cl constr}
\noindent(a) Start with ${\bf x}$ and choose, if possible, one $i$ such that
$x_i\in S$ and $y_i\notin S$. If there exists no such $i$, then the algorithm
terminates.

\noindent(b) Consider the cluster ${\bf x}_i=(x_1,\ldots, \hat x_i,y_i,\ldots, x_n)$ and the set $S_i$, defined for
${\bf x}_i$, just as  $S$ is defined for ${\bf x}$. 

\noindent(c) Repeat Step (a) with ${\bf x}={\bf x}_i$.
\end{algorithm}

The algorithm terminates after at most $m$ iterations, and we obtain a cluster
${\bf x'}$ such that $\II\cap Y'$ is defining. 

From now on, we may therefore assume that ${\bf x}$ is defining for a TPP $\II$ with $\II\cap Y=S$.

\begin{lemma}
  The ideal $\II$ is a minimal prime over $I_S$.
\end{lemma}

\begin{proof}
Suppose that $\II$ is not minimal over $I_S$ and let $\II'$ be a minimal prime over $I_S$ such that $\II\supset I_S$. Then $dim(X_\II)<  dim(X_{\II'})\le|S\cap\{\bf x \}|$. This, however, contradicts  Proposition \ref{pr:dimension form}. The lemma is proved.
\end{proof}

Now, let ${\bf x}$ be any cluster and  $Y$ as above. There are only finitely many clusters ${\bf x}_1,\ldots,{\bf x}_p$ that one can reach from ${\bf x}$ with at most $m$ mutations, each of which is geometrically  generic by our assumptions.  Denote by $Y_1,\ldots, Y_p$ the corresponding sets "$Y$".  Each TPP is a minimal prime over some $I_{T}$ where $T\subset Y_h$ for some $1\le h\le p$.   The union of the power sets of $Y_1,\ldots, Y_p$ is finite. We assumed $\AA$ to be Noetherian, hence there exist only finitely many minimal prime ideals over any given ideal $I_{T}$. Proposition \ref{th:finiteness} is proved.
\end{proof}
 
 Hence, we have proved Theorem \ref{th:Stratification total}, in the following  precise form.
 
 \begin{theorem}
\label{th:Stratification total 2} 
Let $\mathfrak{A}$ be a Noetherian Poisson cluster algebra  over the complex numbers, corresponding to the triple
$({\bf x}, B,\Lambda)$ introduced above.  If  the  seed is super-toric and  geometrically $m$-generic,  then  $\AA$  contains only finitely many Poisson prime ideals which are invariant under the global toric actions.  
\end{theorem}

 \subsection{Torus invariant Poisson prime ideals in $\CC[G(2,5)]$}
The set $Y$ we have to consider is 
$$Y=\{   \Delta_{13} ,\Delta_{24},\   \Delta_{14},\
\Delta_{35},\ \quad \Delta_{12},\ \Delta_{23},\ \Delta_{34},\ \Delta_{45},\
\Delta_{15}, \} \ .$$
The TPPs that have co-dimension one are generated by
the coefficients $\Delta_{12}$, $\Delta_{23}$, $\Delta_{34}$, $\Delta_{45}$,
$\Delta_{15}$. Notice that the cluster variables which are not coefficients cannot generate Poisson ideals, as e.g.~
$$\{\Delta_{13},\Delta_{24}\}=2\Delta_{14}\Delta_{23}\ .$$
Now, consider  $S=\{\Delta_{12},\Delta_{23}\}\subset Y$. $S $ does not define a
toric Poisson prime ideal, since an ideal that contains $S$
but neither $\Delta_{13}$ nor $\Delta_{24}$  cannot be prime because of the Pl\"ucker
relation (Equation \ref{eq:Pluecker rel}) 
$$\Delta_{13}\Delta_{24}=\Delta_{14}\Delta_{23}+\Delta_{12}\Delta_{34} \in\II\
.$$
 However, one  easily verifies that  $S_1=
\{\Delta_{12},\Delta_{23},\Delta_{13}\}$,
$S_2=\{\Delta_{12},\Delta_{23},\Delta_{24}\}$ and
$S_3=\{\Delta_{12},\Delta_{23},\Delta_{13},\Delta_{24}\}$ define toric Poisson
prime ideals.  Observe that the first two have co-dimension two, while the third one has co-dimension three.

\subsection{Varieties with singularities}
\label{se:singularities}

Recall that the cluster manifold defined by Gekhtman, Shapiro and Vainsthein in
\cite[Section 2.1]{GSV} is smooth.  If the variety $X$ defining a cluster
algebra
$\AA=\CC[X]$ is singular, then the cluster manifold will be a smooth
submanifold. Notice, however, that the singular points are the zero locus of a
Poisson ideal (see \cite{Pol}).  Thus, we are able to recover the whole
singular variety, as
the following example shows which is adapted from
an example in \cite{GSV}.

\begin{example}
\label{ex:Singularities}
  Consider a cluster algebra $\AA$ over $\CC$ defined by two clusters
$(x_1, x_2, x_3)$
  and $(x_1', x_2, x_3)$ with exchange relation $x_1 x_1' = x_2^2 + x_3^2$ and
compatible Poisson structure 
$$\{x_1,x_2\}=x_1x_2\ ,\quad   \{x_1,x_3\}=-x_1x_3\ ,\quad \{x_2,x_3\}=0 \ .$$

  We have $\AA \cong \CC [a, b, c, d] / (ab = c^2 + d^2)$, which
defines a
  hypersurface $X$ in $\CC^4$ with a singularity at $a = b = c = d = 0$.
Now let us determine the quotient cluster algebras. It is easy to see that there
are the following toric Poisson ideals (we only list the generators) $\langle x_2\rangle$, $\langle
x_3\rangle$,$\langle x_2,x_3\rangle$,$\langle x_1,x_2,x_3\rangle$,$\langle
x_1',x_2,x_3\rangle$ and $\langle x_1, x_1',x_2,x_3\rangle$. Gekhtman, Shapiro
and Vainsthein show in \cite{GSV} that $X \backslash \{(0, 0, 0, 0)\}$ is the
cluster manifold. In our picture, we observe that the singularity $(0, 0, 0, 0)$ defines a toric Poisson ideal
in $\mathfrak{A}$.  
\end{example}

\section{Acyclic Cluster Algebras}
\label{se:acyclic cl}
We will now apply Theorem \ref{th:Stratification total} to the case of certain acyclic cluster algebras, namely those which have a seed  $({\bf x},B)$ where $B$ is a skew-symmetric $n\times n$-matrix with $b_{ij}>0$ if $i<j$. 
Berenstein, Fomin and Zelevinsky proved in \cite{BFZ} that such a  cluster algebra $\AA$ is equal to both its lower and upper bounds. Thus, it is  Noetherian and, if $B$ has full rank, a Poisson algebra with the Poisson brackets given by compatible  pairs  $(B,\Lambda)$ with $\Lambda=\mu B^{-1}$ for certain $\mu\in \ZZ$. In order for $B$ to have full rank we have to assume that $n=2k$ is even. Finally, notice that there are no global toric actions as $ker(B)$ is trivial. Hence all Poisson prime ideals are TPPs.  Let $P_i=m_i^+ +m_i^-$ where $m_i^+$ and $m_i^-$ denote the monomial terms in the exchange polynomial. Then $\{y_i,x_i\}=\mu_1m_i^+ +\mu_2m_i^-$ for some $\mu_1,\mu_2\in \ZZ$.  We, additionally, want to require that $\mu_1\ne \mu_2$. To assure this, we assume that

\begin{equation}
\label{eq:poisson gen}
\sum_{j=1}^n (b^{-1})_{ij}\left(max(b_{ij},0)+min(b_{ij},0)\right)\ne 0\ 
\end{equation}  
for all $i\in [1,n]$. We have the following result.
\begin{theorem}
\label{th:acyclic}
Let $\AA$ be an acyclic cluster algebra over $\CC$ with trivial coefficients of even rank $n=2k$, given by a seed $(x_1,\ldots, x_{n}, B)$ where $B$ is a skew-symmetric $n\times n$-integer matrix satisfying  $b_{ij}>0$ if $i<j$ and suppose that  $B$ and $B^{-1}$ satisfy Equation \ref{eq:poisson gen} for each $i\in[1,n]$. Then, the Poisson cluster algebra  defined by a compatible pair $(B,\Lambda)$ where $\Lambda=\mu B^{-1}$  with  $0\ne \mu\in \ZZ$ contains no non-trivial Poisson prime ideals.
\end{theorem}

\begin{proof}
 Suppose that there exists a non-trivial TPP or Poisson prime ideal $\II$. Then, $\II\cap {\bf x}$ is nonempty by Proposition \ref{pr:TPP super toric}, hence $\II\cap {\bf x}=\{x_{i_1},\ldots, x_{i_j}\}$ for some $1\le i_1\le i_2\le \ldots\le i_j\le 2k$. Note that ${\bf x}$ does not need to be defining for the ideal $\II$.  Observe that if $b_{i_1,h}<0$, then $x_h\notin \II$ for all $1\le h\le n$. Additionally, observe that   $P_{i_1}=m_{i_1}^++m_{i_1}^-$ has to be contained in $\II$, as well as 
 $$\{y_{i_1},x_{i_1}\}=\mu_1m_{i_1}^+ +\mu_2m_{i_1}^-\ .$$ 
 By our assumption, we have $\mu_1\ne \mu_2$, and therefore $m_{i_1}^-\in \II$.   Hence, $x_h\in\II$ for some $h\in[1, i_{1}-1]$  or $1\in \II$. We obtain the desired contradiction and the theorem is proved.
\end{proof}

   The theorem has the following corollary which was also independently proved by Muller very recently \cite{Mu 1}, though in more generality.
\begin{corollary} Let $\AA$ be as in Theorem \ref{th:acyclic}. Then, the variety $X$ defined by $\AA=\CC[X]$ is smooth. 
\end{corollary}   
   \begin{proof}
   The singular subset is contained in  a Poisson ideal of co-dimension greater or equal to one (see Section \ref{se:singularities}). The Poisson  ideal must be contained in a proper Poisson prime ideal by Lemma \ref{le:Poisson core primes}. 
   The assertion follows.
   \end{proof}
   
\begin{remark}
The assumption that the cluster algebra has even rank is very important. Indeed, Muller has recently shown that the variety corresponding to the cluster algebra of type $A_3$ has a singularity (\cite[Section 6.2]{Mu}).
\end{remark}   

\begin{remark}
 We believe that our results also extend to the locally acyclic cluster algebras introduced in \cite{Mu 1}. However, it should be possible to show that the variety has additionally the structure of a (holomorphic) symplectic manifold.
\end{remark}   

  \section{Explicit Description of Ideals and COS}
  \label{se:COS}
  In the following section we make some additional assumptions about the cluster algebras in questions and their TPPs. In particular, we assume  the following {\it COS} Condition. We use the terminology of a defining cluster, introduced in Section \ref{se:finiteness and defining clusters}. We will also refer to a Noetherian Poisson cluster algebra, simply as a cluster algebra.
  
  \begin{condition}{\bf Codimension One Strata (COS)}\label{cond:COS}
  Let $\AA$ be a  cluster algebra of rank $n$, resp.~upper cluster algebra or upper bound. We say that $\AA$ satisfies COS if for each $1\le k\le \ell\le n$ and each pair of  TPPs $\II\subset \II'$ of co-dimension $n-\ell$, resp. $n-k$, there exists a chain of of TPPs $\II=\II_0\subsetneq \II_1  \subsetneq \ldots \subsetneq \II_{k-\ell}=\II'$.
  \end{condition}

  \begin{remark}
  Once again, we will ignore the case of upper cluster algebras, resp.~upper bounds, however, the arguments are analogous.
  \end{remark}
 Notice that the condition implies the following.

  \begin{lemma}
  \label{le:COSD}
  Let $\AA$ be a cluster algebra, satisfying COS with a geometrically $m$-generic cluster ${\bf x}$ . For each pair of TPPs $\II\subsetneq \II'$ there exists a defining cluster ${\bf x}'$ and a sequence $(x_{i_1}',x_{i_2}',\ldots, x_{i_k}')\in {\bf x}'$, as well as TPPs $\II=\II_0\subsetneq\II_1\subsetneq \ldots\subsetneq \II_k=\II'$ for which ${\bf x}'$ is defining, such that for all $j\in [1,k]$
  \begin{equation}
  \label{eq:cardinalities} 
  \II_j\cap{\bf x}'=(\II\cap{\bf x})\cup\{x_{i_1}',x_{i_2}',\ldots, x_{i_j}'\}\ .
  \end{equation}
  \end{lemma}

\begin{proof}
The assertion follows from the following observation: Let $\II$ and $\JJ$ be two TPPs such that $\II\subset \JJ$ and let ${\bf x}$ be an $m$-generic cluster. We can construct a cluster that is defining for both $\JJ$ and $\II$ by first constructing a defining cluster ${\bf x"}$ for $\II$ using Algorithm \ref{alg:q-cl constr} and, afterwards constructing a defining  cluster for $\JJ$. Indeed, if $x_i"\in \JJ$ but $y_i"\notin \JJ$ for some $x_i"\in {\bf x}"$, then  $x_i"\notin \II$ and $y_i"\notin\II$, because ${\bf x}"$ is defining for $\II$. Hence, all the clusters obtained while constructing a defining cluster for $\JJ$ from ${\bf x}"$ will also be defining for $\II$. Equation \ref{eq:cardinalities} now follows from Proposition \ref{pr:defining clusters dim} and the fact that all clusters constructed are geometrically generic as this implies that $\dim(\AA/\II)=|{\bf x}"-\II|$, resp.~$\dim(\AA/\JJ)=|{\bf x}"-\JJ|$. The lemma is proved.   
\end{proof}
  
  Now, let $\AA$ be a cluster algebra, $\II$ a TPP and ${\bf x}$ a defining cluster. Let $\II\cap {\bf x}=\{x_{i_1},x_{i_2},\ldots, x_{i_k}\}$ and let $\{x_{j_1},\ldots x_{j_\ell}\}={\bf x}-\II$.
Let $\{j_1,\ldots, j_p\}$ be a $p$-element subset of $\{1,\ldots,n\}$.   Notice that the $p\times p$-submatrix $\Lambda_{j_1,\ldots, j_p}$ of the Poisson coefficient matrix $\Lambda$ spanned by the rows and columns labeled by $\{j_1,\ldots, j_p\}$ defines a Poisson bracket on $\CC[x_{j_1}^{\pm 1},\ldots x_{j_p}^{\pm 1}]$.
 We have the following main result.
 
 \begin{theorem}
 \label{th:ideals descr}
 (a) Let $\AA$ be as above satisfying  COS, and let ${\bf x}$ be a defining cluster for the TPP $\II$ and let ${\bf x}$ be geometrically generic and super-toric. Let $z=\sum_{\alpha\in \ZZ^n} c_\alpha x^\alpha\in \AA$, where $c_\alpha\in \CC$. We have $z\in \II$ if and only if $c_\alpha\ne 0$ implies that $\alpha_{i_j} \ne 0$ for some $j\in [1,k]$.   
 
\noindent (b)There exists an injective Poisson algebra homomorphism 
 $$\AA/\II\hookrightarrow \CC[x_{j_1}^{\pm 1},\ldots x_{j_\ell}^{\pm 1}]$$ with Poisson bracket given by $\Lambda_{j_1,\ldots, j_\ell}$,  which sends the image of $x_{j_r}$ in $\AA/\II$ to  $x_{j_r}\in \CC[x_{j_1}^{\pm 1},\ldots x_{j_\ell}^{\pm 1}]$ for all $1\le r\le \ell$.
 \end{theorem} 
 \begin{remark}
 Notice that it is not at all clear that the set of part(a) should even be an ideal.
 \end{remark}

 \begin{proof}
 We prove the assertion by induction on $k$. It is trivially satisfied for $k=0$. Suppose now that the theorem   holds for all TPPs for which the intersection with a defining cluster has cardinality less  than $k-1$. In order to simlify notation, and using the fact that the cluster is geometrically generic, we suppose that $\II_{k-1}\cap{\bf x}=\{x_{\ell+1},x_{\ell+2},\ldots, x_{n}\}$ and $\II_k\cap{\bf x}=\{x_{\ell},x_{\ell+1},\ldots, x_{n}\}$(this might imply that we have to reorder the cluster variables in such a way that not all the coefficients are the "last" indices).
We now make the  following claim.

\begin{claim}There are injective homomorphisms   of algebras
$$\AA/\II_{k-1}\hookrightarrow  \CC[x_{1}^{\pm 1},\ldots, x_\ell^{\pm 1}, x_{\ell+1}]\subset \CC[x_{1}^{\pm 1},\ldots, x_{\ell+1}^{\pm1}]\ .$$
\end{claim} 
 
 \begin{proof} The second inclusion is trivial. Let us prove the first one.
 Suppose not. Then there exists $z\in \AA/\II_{k-1}$ which can be expressed as $z=x_{\ell+1}^{-k}\sum_{\alpha\in \ZZ^{\ell+1}} c_\alpha x^\alpha$ , where $k\in \ZZ_{>0}$, $c_\alpha\in \CC$ and where $k$ is minimal with the property that $c_\alpha\ne 0$ implies that $\alpha_{\ell+1}\ge 0$. If we multiply by the smallest common denominator (a monomial  $x_{\ell+1}^kx_1^{\beta_1}\ldots x_{\ell}^{\beta_{\ell}}$ with $\beta_1,\ldots, \beta_{\ell}\in \ZZ_{\ge0}$), then we obtain an element $\tilde z\in  \CC[x_1,\ldots,x_{\ell+1}]\subset\AA/\II_{k-1}$. Clearly $\tilde z\in \II_k\subset \AA/\II_{k-1}$, where we abuse notation and denote by $\II_k$ the image of the TPP $\II_k\subset \AA$ in $\AA/\II_{k-1}$. The element $\tilde z$ contains at least one monomial summand $c_\gamma x^{\gamma}$, $\gamma\in \ZZ_{\ge 0}^n$ where $\gamma_{\ell+1}=0$. We obtain that its pre-image $\tilde z$ under the projection map $\AA\to \AA/\II_{k-1}$ (which is the identity map on $\CC[x_1,\ldots, x_{\ell+1}]$) lies in $\II_k$, and, hence, as ${\bf x}$ is super-toric, there exists, by Proposition \ref{pr:TPP super toric +}, a cluster variable $x_i$ with $i\le \ell$ such that $x_i\in \II_k$. That however, contradicts our assumption. The claim is proved.
 \end{proof}

Now, consider the ideal generated by $x_{\ell+1}$ in $\CC[x_{1}^{\pm 1},\ldots, x_{\ell}^{\pm 1}, x_{\ell+1}]$ with Poisson structure defined by $\Lambda_{j_1,\ldots, j_{\ell+1}}$. It is easy to see that it is  torus invariant, Poisson and prime. Consider its intersection $\tilde \II$ with $\AA/\II_{k-1}$. It suffices to show that it is the unique minimal toric Poisson prime ideal containing the ideal $\hat \II$  generated by $x_{\ell+1}$,  but none of the $x_1,\ldots x_\ell$. It  is, clearly, toric, Poisson and prime since it is the intersection of a  toric  Poisson prime ideal  with a torus invariant subring. Now, let $z\in \tilde\II-\hat \II$. Then, there exists a monomial $m=x_1^{\alpha_1}\ldots x_\ell^{\alpha_\ell}$ with $\alpha_1,\ldots \alpha_\ell\in \ZZ_{\ge 0}$ (e.g.~the smallest common denominator) such that $m z\in \hat\II$.  Hence any TPP $\JJ$ not containing  any of the $x_1,\ldots, x_\ell$ must contain $z$. This implies that $\tilde \II\subset \JJ$, and hence, $\tilde \II$ is the unique minimal toric Poisson prime with  $\tilde \II\cap{\bf x}=\{x_{\ell+1},\ldots, x_n\}$. Part (a) is proved and part (b) follows from the fact that if $R$ is a Poisson algebra, $S\subset R$  a Poisson subalgebra and $I\subset R$ a Poisson ideal, then $I\cap S$ is a Poisson ideal in $S$, and the canonical inclusion $S/(I\cap S) \hookrightarrow R/I$ is a homomorphism of Poisson algebras. The theorem is proved.
 \end{proof}
 
 Theorem \ref{th:ideals descr} and Lemma \ref{le:COSD} imply the following corollary.
 
 \begin{corollary}
 \label{co:incl}Let $\AA$ be a cluster algebra and suppose that all clusters are geometrically generic and super-toric. 
 Let $\II$, and $\JJ$ be two TPPs. Then, $\II\subset \JJ$ if and only if there exists a cluster ${\bf x}$, defining for both $\II$ and $\JJ$,  such that $(\II\cap{\bf x})\subset (\JJ\cap {\bf x})$.
 \end{corollary}

 \subsection{Cluster Algebras satisfying COS and the COS Conjecture}
 
 Theorem \ref{th:ideals descr} applies in many important cases. It is, for example, well known that the stratification of a complex semisimple connected and simply connected algebraic group $G$ with the standard Poisson structure, into double Bruhat cells provides a stratification that satisfies COS (see e.g. \cite{BFZ} or \cite{bz-qclust}). The double Bruhat cells $G^{u,v}$ are labeled by double words $u,v\in W$ where $W$ denotes the Weyl group of $G$. The dimension of $G^{u,v}$ is $\ell(u)+\ell(v)+r$ where $\ell(w)$ denotes the length of an element $w$ of $W$. The zero loci of the corresponding TPP $J^{u,v}$ are the   double Bruhat cells $G^{u',v'}$ where $u'\le u$ and $v'\le v$ in the Bruhat order. Indeed, each double Bruhat cell has an upper  cluster algebra structure by \cite{BFZ}, however, it is not known how to relate cluster algebra structures of different double Bruhat cells. Now, let $u,v\in W$ and let ${\bf u}$,${\bf v}$ be reduced expressions of $u$ and $v$,respectively. Let ${\bf w_0}$ and ${\bf w_0}'$ be reduced expressions of the longest element of the Weyl $w_0$ such that ${\bf w_0}={\bf u}{\bf u}'$ and ${\bf w_0}'={\bf v}{\bf v}'$(i.e.~the reduced expressions start with ${\bf u}$, resp.~${\bf v}$).  It is now easy to see that in the corresponding cluster consisting of generalized minors (see \cite{BFZ}) we obtain the description of $J^{u,v}$ of Theorem \ref{th:ideals descr}(a) and the Poisson homomorphism of its part (b). Clearly, it will be the next important step to prove that the assumptions (super-toric and geometrically generic) also apply to other clusters. 
 
We have  a similar story, when $\gg$ is  a symmetric Kac-Moody Lie algebra, and $W$ its Weyl group. The unipotent radicals $U_w$, associated to each $w\in W$, together with the standard Poisson structure, admit   a stratification (see e.g. \cite{Yak-nil}) satisfying COS. Indeed  the strata are labeled by the elements $v\in W$ such that $v\le w$. And Geiss, Leclerc and Schr\"oer proved that these algebras have a cluster algebra structure (\cite{GLS2}).  We believe, again, that each cluster satisfies the conditions of Theorem \ref{th:ideals descr}.

Of course, COS is trivially satisfied in the case when $\A$ is an acyclic cluster algebra without coefficients with exchange matrix of full rank (see Section \ref{se:acyclic cl}).
This motivates the following conjecture.

\begin{conjecture}
\label{con:COS}
Let $\AA$ be a Noetherian Poisson (upper) cluster algebra with an exchange matrix of full rank. Then $\AA$ satisfies COS. 
\end{conjecture}

 Moreover, computations for a number of examples and Conjecture \ref{con:COS} suggest the following very strong statement.
 
 \begin{conjecture}
 \label{conj:poisson and prime}
Let $\AA$ be a Noetherian Poisson (upper) cluster algebra  with exchange matrix of full rank. Let $\II$ be a TPP and  suppose that ${\bf x}$ is a defining cluster for $\II$ with $x_i\notin \II$ for some $1\le i\le n$. There exists a TPP $\JJ$ such that $\JJ\cap{\bf x}=(\II\cap {\bf x}) \cup \{x_i\}$ if and only if $P_i\in \II$ or $x_i$ is a coefficient. 
 \end{conjecture}

\begin{remark}
The "only if" direction is of course Theorem \ref{th:Stratification total}. The "if" part would allow us to compute the stratification simply using the matrix $B$ and its mutations without having to know more about the algebra.
\end{remark}

\begin{appendix}

\section{Toric Prime Ideals}
\label{se:toric primes}
In this appendix we review some facts regarding prime ideals stable under the
action of some torus, essentially following the discussion by Brown and Goodearl
in
\cite[Ch.II]{brown-goodearl}, and similar to \cite{Good}. Let $H$ be a group
acting by automorphisms on a
ring
$R$. An ideal $I$ of $R$ is called $H$-stable if $h (I) = I$ for all $h \in
H$. For convenience we will write $H$-ideals to denote $H$-stable ideals. We
say that $R$ is {\it $H$-prime} if $R$ is nonzero and
the product of two non-zero $H$-ideals is non-zero. An $H$-prime ideal of $R$ is
any proper $H$-ideal $I$ such that $R / I$ is an $H$-prime ring. For any ideal
$I$ in $R$ we denote by $(I:H )$ the largest $H$-ideal containing $I$, i.e.
\[ (I : H) = \bigcap_{h \in H} h (I) . \]
Note the following facts.

\begin{lemma}
  \label{le:large of prime Hprime}{\cite{brown-goodearl}}[II.1.9] Let H be a
  group acting by automorphisms on a ring $R$. If $P$ is a prime ideal, then
  (H:P) is an $H$-prime ideal. 
\end{lemma}

\begin{lemma}
\label{le:H primes are prime}
  {\cite{brown-goodearl}}[II.1.12] Let R be a Noetherian ring, and suppose that
  a $k$-torus $H$ acts on $R$ by automorphisms. If $k$ is algebraically
  closed, then all $H$-prime ideals of $R$ are prime.
\end{lemma}

We derive the following corollary.

\begin{lemma}
\label{le:primes are H primes}
  Let $R$ be a Noetherian ring, and suppose that a $k$-torus $H$ acts on $R$
  by automorphisms and that $k$ is algebraically closed. Let $I$ be an
  H-ideal and let $P$ be a minimal prime over $I$. Then $P$ is an $H$-prime.
\end{lemma}

 \begin{proof} Suppose not. Then $I \subset (P: H) \subset P$. Lemma
\ref{le:large of prime Hprime} 
yields that $(P:H)$ is an $H$-prime ideal, while one derives from
Lemma \ref{le:H primes are prime}  that $(P : H)$ is a prime ideal. One
concludes that
$P=(P:H)$ since $P$ was assumed to be minimal. The lemma is proved.
\end{proof}

Note that we have not required that $(P:H)$ be Poisson. We will call a Poisson
structure on a $k$-algebra $A$ {\it compatible} with the action of the torus $H$
if $H$ acts by Poisson automorphisms, that means
$$ \{h(x),h(y)\}=h(\{x,y\})\ .$$
We have the following fact.
\begin{lemma}
\label{le:H Poisson}
Let $(A,\{\cdot,\cdot\})$ be a Poisson algebra and suppose that a $k$-torus $H$ acts
on $A$
  compatibly. If $I$ is a Poisson ideal, then $(I:H)$ is a Poisson
 ideal as well.
\end{lemma}

\begin{proof}
Notice that $h(I)$ is a Poisson ideal for all $h\in H$. Indeed, we obtain for
all $x\in I$ and $y\in A$
$$\{h(x),y\}=\{h(x),h(h^{-1} (y))\}=h(\{x,h^{-1} (y)\})\in h(I)\ .$$
We conclude immediately that $(I:H)$ is a Poisson ideal.
\end{proof}

\end{appendix}

\begin{acknowledgements}
 The author would like to thank T.~Lenagan, S.~Launois, M.~Yakimov and
J.~Schr\"oer for several suggestions and comments, and the referees and Giovanni Cerulli-Irelli for many
suggestions regarding the structure of the paper.
\end{acknowledgements}

\end{document}